\documentclass[12pt]{article}
\usepackage{amsmath}
\usepackage{amsthm,amsfonts,amssymb,epsfig,graphics}
\usepackage[latin1]{inputenc}
\usepackage[francais,english]{babel}
\usepackage{pstricks,pst-plot}

\def\ra{\rightarrow}

\def\adhe{\mathrm{Clos}}
\def\inte{\mathrm{Int}}

\newcommand{\bbR}{{\mathbb{R}}}
\newcommand{\bbQ}{{\mathbb{Q}}}
\newcommand{\bbC}{{\mathbb{C}}}
\newcommand{\bbZ}{{\mathbb{Z}}}

\def\cC{{\cal C}}

\def\?{$^{***}$\marginpar{?}}

\newtheorem{theo}{Theorem}
\newtheorem*{theo*}{Theorem}
\newtheorem{prop}[theo]{Proposition}

\newtheorem*{conj*}{Conjecture}
\newtheorem{lemma}[theo]{Lemma}
\newtheorem{claim}[theo]{Claim}

\theoremstyle{definition}
\newtheorem{defi}[theo]{Definition}

\newtheorem{rema}[theo]{Remark}

\newtheorem*{rema*}{Remark}
\newtheorem*{remas*}{Remarks}

\author{Fr\'ed\'eric Le Roux \\
Universit\'e Paris-Sud, Bat. 425 \\
91405 Orsay Cedex   FRANCE   
}
\date{}

\begin{document}
\sloppy

\title{A topological characterisation of holomorphic parabolic germs in the plane}

\maketitle

\begin{abstract}
In~\cite{GP}, Gambaudo and P\'ecou introduced  the ``{linking property}''
to study the dynamics of germs of planar homeomorphims and provide a new proof of Naishul theorem.
In this paper we prove that the negation of Gambaudo-P\'ecou property characterises the topological dynamics of holomorphic parabolic germs.
As a consequence, a rotation set for germs of surface homeomorphisms around a fixed point can be defined, 
and it will turn out to be non trivial except for countably many conjugacy classes.
\end{abstract}

\section{Introduction}
\def\hp{{\cal H}^+}
Let $\hp$ be the set of  orientation preserving homeomorphisms of the plane that fix $0$, and let $h \in \hp$.
We are interested in the dynamics of the germ of $h$ at $0$.
Imagine one wants to evaluate the ``amount of rotation'' in a neighbourhood $V$  of $0$ by looking at the way the orbit of some point $x \in V$ rotates around  $0$. Then two kinds of difficulties can arise:
\begin{itemize}
\item if the orbit of $x$ leaves $V$ after a small number of iterations, then 
the behaviour of $x$ is not significant with respect to the local dynamics;
\item if the orbit of $x$ tends to the fixed point $0$, then the rotation of $x$ around $0$
is not significant either, because it is not invariant under a continuous change of coordinates. 
\end{itemize}
These difficulties has lead Gambaudo and P\'ecou to the statement of the ``{linking property}'' (see~\cite{GP,Pe}) which demands that inside each neighbourhood of $0$ there exist arbitrarily long segments of orbits starting and ending not too close to $0$.
In this paper we prove that the only germs that do not share the linking property are  the contraction, dilatation and holomorphic parabolic germs. 
To be more precise, let us define the \emph{short trip property}, which is the negation of Gambaudo-P\'ecou property, as follows.

\begin{defi}\label{def.property}
Let  $f\in \hp$. 
We say that $f$ satisfies the \emph{short trip property} if 
there exists a neighbourhood $V$ of the fixed point $0$, such that  for every 
neighbourhood $W$ of $0$,
there exists an integer $N_{W} >0$ such that
for every segment of orbit $(x,f(x), \dots, f^n(x))$ which is included in  $V$, 
and whose endpoints $x, f^n(x)$ are outside $W$, the length $n$ is less than $N_{W}$.
\end{defi}

\pagebreak[1]
Two homeomorphisms $f_{1},f_{2} \in \hp$ are said to be \emph{locally topologically conjugate} if there exists a homeomorphism $\varphi \in \hp$ such that the relation $f_{2} = \varphi f_{1} \varphi^{-1}$ holds on some neighbourhood of $0$.
We are interested in the local dynamics near the fixed point $0$, thus we consider 
maps up to  local conjugacy.
Note that any local homeomorphism locally coincides with a  homeomorphism defined on the whole plane, so that working with globally defined homeomorphisms is just a matter of convenience and does not alter the results (see~\cite{Ham} or~\cite{LR1}, chapitre 2). As a consequence, to prove that two homeomorphisms are locally topologically conjugate it suffices to construct the conjugacy on a neighbourhood of $0$.
\begin{defi}
Let $f \in \hp$, and identify the plane with the complex plane $\bbC$. We say that $f$ is a \emph{locally holomorphic parabolic homeomorphism} (or just \emph{parabolic})
if $f$ is holomorphic on some neighbourhood of $0$, $f'(0)$ is a root of unity,
and  for every positive  $n$ the map   $f^n$ is not locally equal to the identity. 
\end{defi}
Note that the hypothesis on $f'(0)$ amounts to saying that the differential of $f$ is a rational rotation, and  then the last hypothesis is equivalent to saying that $f$ is not locally topologically conjugate to its differential. According to Camacho version of Leau-Fatou theorem, if $f \in \hp$ is  parabolic, then $f$ is locally topologically conjugate to some map
 $$
 z \mapsto e^{2i\pi \frac{p}{q}}z(1 + z^{qr}), \ \ \ \mbox{ with } \frac{p}{q} \in \bbQ, q,r\geq1.
 $$
 See ~\cite{Cam}, and figure~\ref{fig1}.

\def\JPicScale{1}
\begin{figure}[htbp]
\begin{center}
\input{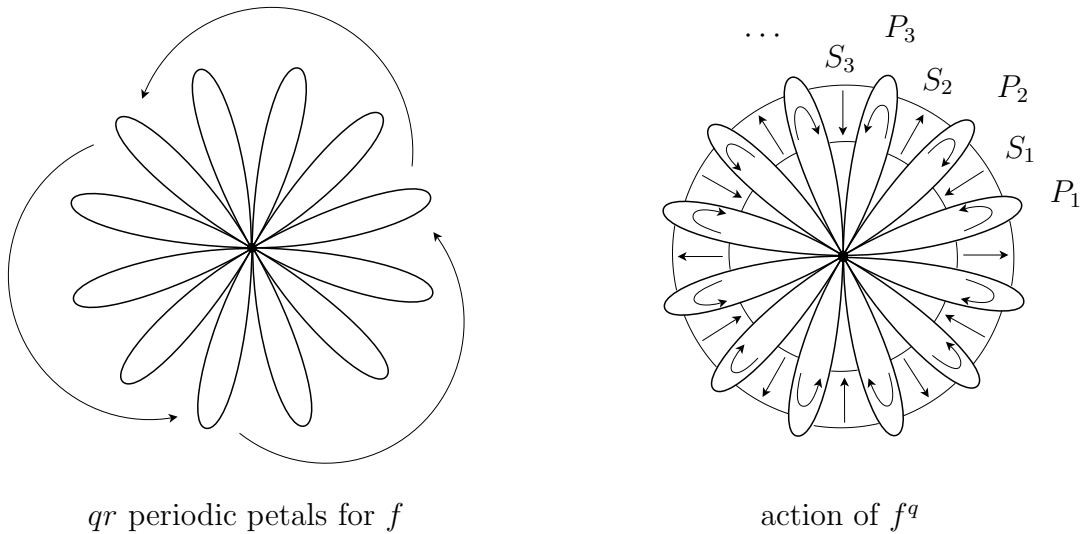}
\caption{Local topological dynamics of $f:z \mapsto e^{2i\pi \frac{p}{q}}z(1 + z^{qr})$; here $q=3$, $p=1$, $r=2$, so that there is two orbits of attracting petals and two orbits of repulsive petals}
\label{fig1}
\end{center}
\end{figure}

We can now state our theorem. 
\begin{theo}\label{theo.main}
Let  $f$ be an orientation preserving homeomorphism of the plane that fixes the point $0$. 
Then $f$ has the short trip property if and only if it is locally topologically conjugate to the  contraction $z \mapsto \frac{1}{2}z$, to the dilatation $z \mapsto 2z$, or
to a locally holomorphic parabolic homeomorphism.
\end{theo}
As a consequence there are only  countably many conjugacy classes missing to satisfy Gambaudo-P\'ecou property.

In order to explain where theorem~\ref{theo.main} comes from, let us first discuss Naishul theorem.
In~\cite{GP} it was shown that Gambaudo-P\'ecou property holds when $f$ preserves area,
and then this property is used to prove Naishul theorem: \emph{among area preserving homeomorphisms fixing $0$ that are differentiable at $0$ and whose differential is a rotation, the angle of the rotation is invariant under a local topological conjugacy}.
Then  the following nice generalisation of Naishul theorem is given by Gambaudo, Le Calvez and P\'ecou in~\cite{GLP}.
As a generalisation of differentiability at $0$, they consider the homeomorphisms $f$ for which the fixed point can be ``blown-up'',
\emph{i.e.} replaced by an ideal  circle in such a way that $f$ can be extended to a circle homeomorphism (see the precise definition in~\cite{GLP}). They prove that for such homeomorphisms, \emph{the Poincar\'e rotation number of the extended circle homeomorphism is invariant under a local topological conjugacy, unless $f$ is a contraction or a dilatation}.
The strategy of their proof is the following.
If $f$ has Gambaudo-P\'ecou property, then one can use  the arguments in~\cite{GP}.
Now assume that  $f$ is \emph{indifferent}, that is, $f$ admits arbitrarily small non-trivial invariant compact connected sets $K$ containing $0$; then one can  use Caratheodory prime ends theory to associate a circle homeomorphism $f_{K}$ to each such $K$,  and use  the rotation number of $f_{K}$ to prove the topological invariance.
 Then one proves a last lemma asserting that \emph{a germ which is not indifferent and does not have Gambaudo-P\'ecou property  must be a contraction or a dilatation}.

As a consequence of Leau-Fatou theorem,  parabolic maps are indifferent. Thus theorem~\ref{theo.main} is a generalisation of this last  lemma.
Furthermore, it  provides an alternative proof of the generalised Naishul theorem avoiding the use of prime-ends, as follows:
we keep the arguments in~\cite{GP} to tackle homeomorphisms with Gambaudo-P\'ecou property;
then, in view of theorem~\ref{theo.main}, it only remains to deal with parabolic homeomorphisms, for which the proof is easy because
the local dynamics is completely understood.

More generally, in~\cite{LR2} we will  define a local rotation set for any homeomorphism $f$ in $\hp$.  This set is a subset of the line modulo integer translation, and it is a local topological conjugacy invariant.
Then  theorem~\ref{theo.main} will entail that the local rotation set  is non void as soon as $f$ does not fall into the countably many conjugacy classes described by the theorem.

One can also think of theorem~\ref{theo.main} as a local analogue of previous results
showing that a simple dynamical property can imply a strong rigidity.
The most striking result here is probably Hiraide-Lewowicz theorem  that an expansive homeomorphism on a compact surface is conjugate to a pseudo-Anosov homeomorphism (see~\cite{Hi,Le}).
Closer to our setting, K\'er\'ekj\'art\'o has shown that an orientation preserving homeomorphism of a closed orientable surface whose singular set is totally disconnected is topologically conjugate
to a conformal transformation (see~\cite{BK,Ke,Ke3}).
Thus, for instance, an orientation preserving homeomorphism $f$ of the plane is conjugate to a translation if and only if it has no fixed point and the family $(f^n)_{n \geq 0}$ is equicontinuous at each point for the spherical metric.

In some sense, theorem~\ref{theo.main} highlights that it is easy to be locally conjugate to 
a locally parabolic homeomorphism: a homeomorphism that ``looks like'' a parabolic map will be conjugate to it.
In contrast, the examples given in~\cite{BLR}
reveal how difficult it is to be conjugate to the saddle homeomorphism $(2x,y/2)$, and in particular that it is not enough to preserve the hyperbolic foliation. 
A topological characterisation can be given, but it must take into account the sophisticated  \emph{oscillating set} (see the remark on fig.~3 in~\cite{BLR}, as well as  part III).

\section{Dynamics of parabolic germs}
Properties~\ref{prop.carac1} and~\ref{prop.carac2} below provide a first (classical) characterisation of parabolic germs in terms of \emph{attracting}  and \emph{repulsive sectors} and \emph{invariant petals}.

\subsection{Contractions and attracting sectors}
We begin by characterising the dynamics of contractions. 
Then we describe \emph{attracting sectors}.
Of course, similar results hold for dilatation and \emph{repulsive sectors}, although we will not state them explicitly.

Let $f\in\hp$. We will say that a sequence $(E_{n})_{n \geq 0}$ of subsets of the plane  \emph{converges to $0$} if for every $W$ neighbourhood of $0$, all but finitely many terms of the sequence are included in $W$. The following result is very classical.
\begin{prop}\label{prop.contraction}
Let $f \in \hp$. Let $D$ be a topological  closed disc\footnote{A \emph{topological closed disc} is a set homeomorphic to the closed unit disc.} 
 which is a neighbourhood of $0$,
and suppose that the orbit $(f^n(D))_{n \geq 0}$  converges to $0$.
Then $f$ is locally topologically conjugate to the contraction $z \mapsto \frac{1}{2}z$.
\end{prop}
\begin{proof}
By hypothesis there exists $n>0$ such that $f^n(D) \subset \inte(D)$.
Choose some decreasing finite sequence of topological closed discs $D_{i}$
with $D_{0}=D$, $\inte (D_{i}) \supset D_{i+1}$, and $\inte (D_{n-1}) \supset f^n(D_{0})$.
 Consider the set
 $$
 O = \inte \left(D_{n-1}\right) \cap \inte \left( f(D_{n-2}) \right) \cap \cdots \cap \inte \left( f^{n-1}(D_{0}) \right).
 $$
 Let $U$ be the connected component of $O$ containing the fixed point $0$.
The hypotheses on the $D_{i}$'s entail that $\adhe(f(O)) \subset O$.
Since $\adhe(f(U))$ is connected and contains $0$, we deduce that $\adhe(f(U)) \subset U$.
Furthermore, according to a theorem of K\'er\'ekj\'art\'o,
the set $D'=\adhe(U)$ is a closed topological disc (see~\cite{Ke2, LCY}). This disc satisfies 
 $f(D') \subset \inte (D')$.
 
Now the annulus $D' \setminus \inte(f(D'))$ is a ``fundamental domain'' for $f$, and can be used to construct a local topological conjugacy
between $f$ and the contraction.
%
%
\end{proof}

We will say that two sets $S$ and $S'$ \emph{coincides in a neighbourhood of $0$}, or \emph{have the same germ at $0$}, and we will write $S \overset{0}{=} S'$,
if there exists a neighbourhood  $V$ of $0$ such that $S \cap V = S' \cap V$.
\begin{defi}[see figure~\ref{fig2}]
An \emph{attracting sector} is a topological closed disc $S$ whose boundary contains $0$, which coincides in a neighbourhood of $0$ with its image $f(S)$, and 
whose orbit $(f^n(S))_{n \geq 0}$ converges to $0$.
The attracting sector is said to be \emph{nice} if $f(S) \subset S$ and $S \setminus f(S)$ is connected.
A \emph{(nice) repulsive sector} is a (nice) attracting sector for $f^{-1}$.
\end{defi}

\def\JPicScale{1}
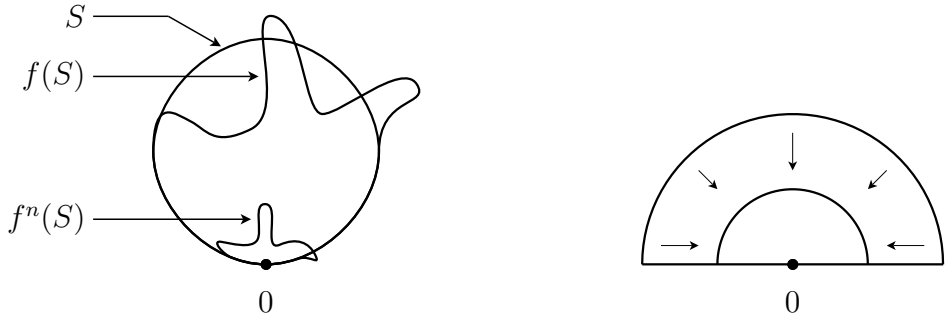
\begin{figure}[htbp]
\begin{center}
\ifx\JPicScale\undefined\def\JPicScale{1}\fi
\psset{unit=\JPicScale mm}
\psset{linewidth=0.3,dotsep=1,hatchwidth=0.3,hatchsep=1.5,shadowsize=1}
\psset{dotsize=0.7 2.5,dotscale=1 1,fillcolor=black}
\psset{arrowsize=1 2,arrowlength=1,arrowinset=0.25,tbarsize=0.7 5,bracketlength=0.15,rbracketlength=0.15}
\begin{pspicture}(0,0)(120,44)
\psline(80,10)(120,10)
\rput{0}(100,10){\parametricplot[arrows=-]{0}{180}{ t cos 20 mul t sin 20 mul }}
\rput{0}(100,10){\parametricplot[arrows=-]{0}{180}{ t cos 10 mul t sin 10 mul }}
\psline[linewidth=0.1]{->}(100,27.5)(100,22.5)
\psline[linewidth=0.1]{->}(87.5,22.5)(90,20)
\psline[linewidth=0.1]{->}(112.5,22.5)(110,20)
\psline[linewidth=0.1]{->}(82.5,12.5)(87.5,12.5)
\psline[linewidth=0.1]{->}(117.5,12.5)(112.5,12.5)
\psdots[](100,10)
(100,10)
\rput(100,5){$0$}
\psdots[](30,10)
(30,10)
\rput(30,5){$0$}
\rput[r](6,43){$S$}
\rput[r](6,35){$f(S)$}
\rput[r](6,16){$f^n(S)$}
\pscustom[]{\psbezier(30,10)(37.5,10)(45,17.5)(45,25)
\psbezier(45,25)(45,32.5)(37.5,40)(30,40)
\psbezier(30,40)(22.5,40)(15,32.5)(15,25)
\psbezier(15,25)(15,17.5)(22.5,10)(30,10)
\closepath}
\pscustom[]{\psbezier(30,10)(37.5,10)(45,17.5)(45,25)
\psbezier(45,25)(45,32.5)(52,31)(50,34)
\psbezier(50,34)(48,37)(40,30)(37.5,30)
\psbezier(37.5,30)(35,30)(35,42)(31,43)
\psbezier(31,43)(27,44)(33,30.5)(28,28)
\psbezier(28,28)(23,25.5)(22,28)(20,29)
\psbezier(20,29)(18,30)(15,32.5)(15,25)
\psbezier(15,25)(15,17.5)(22.5,10)(30,10)
\closepath}
\pscustom[]{\psbezier(30,10)(32,10)(33.1,10.4)(34.5,10.8)
\psbezier(34.5,10.8)(35.9,11.2)(37.3,9.7)(36.5,11.3)
\psbezier(36.5,11.3)(35.7,12.9)(34.1,13.6)(31.8,12.6)
\psbezier(31.8,12.6)(29.5,11.6)(31.87,17.74)(30,18)
\psbezier(30,18)(28.13,18.26)(29.6,14.4)(28.8,13.2)
\psbezier(28.8,13.2)(28,12)(25.3,12.8)(24,13)
\psbezier(24,13)(22.7,13.2)(24.2,11.4)(25.4,10.9)
\psbezier(25.4,10.9)(26.6,10.4)(28,10)(30,10)
\closepath}
\psline[linewidth=0.2]{->}(17,43)(24,39)
\psline[linewidth=0.2]{->}(7,35)(29,35)
\psline[linewidth=0.2]{->}(7,16)(28,16)
\psline[linewidth=0.2](17,43)(7,43)
\end{pspicture}
\caption{Attracting and nice attracting sectors}
\label{fig2}
\end{center}
\end{figure}

\pagebreak[3]
\begin{claim}~ \label{claim2}
\begin{enumerate}
\item If $S$ is an attracting sector  then there exists a nice attracting sector $S'$, included in $S$, and  having the same germ as $S$ at $0$.
\item   If $S'$ is a nice attracting sector for $f$, then there exists a homeomorphism $\Phi$ between $S'$ and the half-disc  $S_{0} = \{|z| \leq 1, y \geq 0\}$
such that the conjugacy relation $\Phi f  = \frac{1}{2} \Phi$ between $f$ and the contraction $z \mapsto \frac{1}{2}z$ holds on $S'$.
\end{enumerate}
\end{claim}

\begin{rema}\label{rema.sector}
Here are some easy consequences of item 2 of the claim.
\begin{enumerate}
\item The sets  $\Phi^{-1}([-1,0]), \Phi^{-1}([0,1])$ are called the 
\emph{sides} of the nice attracting sector; they do not depend  on the choice of $\Phi$.

\item There exists arbitrarily small nice attracting sectors within $S'$; moreover, any pair of points $x,y$ on both sides of $S'$
are the endpoints of the sides of some nice attracting sector included in $S'$.

\item Any homeomorphism $\Phi$ between the union of the sides of $S'$ and the segment $[-1,1]$,
satisfying the conjugacy relation $\Phi f  = \frac{1}{2} \Phi$, can be extended to a homeomorphism between $S'$ and $S_{0}$ conjugating $f$ and $z \mapsto \frac{1}{2}z$ as in item 2 of the claim.
\end{enumerate}
\end{rema}

\begin{proof}[Proof of claim~\ref{claim2}]
Let $S$ be an attracting sector.
It is easy to see that there exists an arc $\alpha$ included in the boundary of $S$, whose interior\footnote{The interior of a curve is defined to be the curve minus its endpoints.} $\inte(\alpha)$
 contains the fixed point $0$, and such that $f(\alpha) \subset \inte(\alpha)$.
Let us consider the set
$$
A:=\bigcup_{n \geq 0} f^{-n}(\alpha).
$$ 
This set is clearly a continuous one-to-one image of the real line, and $f(A)=A$.
By definition of an attracting sector, there exists an integer $n_{0}$ such that
for every $n \geq n_{0}$, the set $f^n(S)$ is disjoint from the compact set $f^{-1}(\alpha) \setminus \inte(\alpha)$.  
Then one has $A \cap S  = f^{-n_{0}}(\alpha) \cap S$.
In particular,  we can find a simple arc $\beta$ such that $\alpha \cup \beta$ is a Jordan curve included in $S$, and
whose intersection with $A$ is reduced to $\alpha$.

Let $D_{0}$ be the topological closed disc bounded by $\alpha \cup \beta$.
Then $D_{0}$ is included in $S$ and coincides with $S$ in a neighbourhood of $0$, 
and $D_{0} \cap A = \alpha$.
Note that for every $n$, $f^n(D_{0}) \cap A = f^n(D_{0} \cap A) = f^n(\alpha)$.
The disc $D_{0}$ is clearly an attracting sector.
 Let $n_{0}$ be a positive integer such that for every $n \geq n_{0}$,
$f^n(D_{0})$ does not meet $\beta$. Thus $f^n(D_{0})$ is included in $\inte(D_{0}) \cup f^n(\alpha)$.

We can now find a closed topological disc $S' \subset D_{0}$, having the same germ at $0$ as $D_{0}$,  which is a nice attracting sector for $f$.
 For this we can make a construction similar to the proof of proposition~\ref{prop.contraction}, with the following adaptations. Now we  choose the topological discs $D_i$'s having the same germ at $0$, containing $\alpha$, and such that 
  $\inte(D_{i}) \cup \{\alpha\} \supset D_{i+1}$ and   $\inte(D_{n-1}) \cup \{f^n(\alpha)\} \supset f^n(D_{0})$. The open set $U$ is defined to be the unique connected component of $O$ that has the same germ at $0$ than the $D_{i}$'s, and $S'$ is the closure of $U$.
  Then one has $f(S') \subset U \cup f^n(\alpha)$, and $f^n(S')$ contains $f^n(\alpha)$; and thus $S' \setminus f(S')$ is connected. 
  The details are left to the reader.

The proof of item 2 is straightforward by using the fundamental domain $\adhe(S' \setminus f(S'))$.
\end{proof}

\subsection{Regular invariant petals}
We first recall a theorem of K\'er\'ekj\'art\'o (\cite{Ke3}).
\begin{theo}[K\'er\'ekj\'art\'o] \label{theo.kerek}
Let $f$ be a homeomorphism of the plane that preserves the orientation, and suppose that for any compact set $K$ the orbit
$(f^n(K))_{n \geq 0}$ converges to the point at infinity in the sphere $\bbR^2 \cup\{\infty\}$.
Then $f$ is topologically conjugate to the translation $z \mapsto z+1$.
\end{theo}
\begin{proof}[Sketch of proof]
By considering the space of orbits, the problem can be brought into the realm of the classification of surfaces:
then it follows from the fact that any surface without boundary, whose fundamental group is the group of integers, is homeomorphic to the infinite cylinder (see for example~\cite{AS}).
\end{proof}

\begin{defi}
An \emph{invariant petal for $f$}
is a topological closed disc  $P$ whose boundary contains the fixed point $0$, and such that $f(P)=P$. An invariant petal is called \emph{regular} if 
 for every compact set $K \subset P\setminus\{0\}$, the sequence $(f^n(K))_{n \geq 0}$ converges to $0$.
\end{defi}

\begin{rema}\label{rem.orientation}
If $P$ is a regular invariant petal then $f$ has no fixed point on the topological line $\partial P \setminus\{0\}$.
Thus we may endow this line with a \emph{dynamical order}  such that $f(x) >x$ for any point $x\neq0$ on $\partial P$. The petal will be called \emph{direct} if  this dynamical  order is compatible  with the topological (usual) orientation of $\partial P$ as a Jordan curve of the oriented plane
(for which the interior of $P$ is ``on the left'' of $\partial P$); in the opposite case it will be called \emph{indirect}.
\end{rema}

An adaptation of the proof of K\'er\'ekj\'art\'o  theorem yields the following.
\begin{claim}
\label{claim1}
Let $P$ be a regular  invariant petal for $f \in \hp$. If $P$ is direct then 
 the restriction $f_{\mid P}$
is topologically  conjugate, \emph{via} an orientation preserving homeomorphism,
to the restriction of the translation $z \mapsto z+1$ to the closed half-sphere  $\{ x+iy, y \geq 0  \} \cup \{\infty\}$ of the Riemann sphere $\hat \bbC$.
The same is true if $P$ is indirect with $z+1$ replaced by $z-1$.
\end{claim}

\subsection{Characterisation of parabolic homeomorphisms}
We can now characterise the local dynamics of parabolic homeomorphisms.
For any set $D$ the \emph{maximal invariant set} of $D$ is the set $\cap_{n \in \bbZ}f^n(D)$ of points whose whole orbits are included in $D$.
\pagebreak[3]
\begin{prop}[see figure~\ref{fig1}]
\label{prop.carac1}
Let $f \in \hp$. Fix some integer $\ell \geq 1$. Then $f$ is locally topologically conjugate to $z \mapsto z(1+z^{\ell})$ if and only if there exists a neighbourhood of $0$ which is a topological closed disc $D$ called a \emph{nice disc}, such that
\begin{enumerate}
\item the maximal invariant set of $D$ is the union of  $2\ell$ regular invariant petals $P_{1}, \dots , P_{2\ell}$, whose pairwise intersections are reduced to $\{0\}$;
\item the sets $\partial D \cap P_{i}$ are connected,
and the cyclic order of these sets along $\partial D$ coincides with the order of the indices $i \in \bbZ/2\ell\bbZ$;
\item for every $i$, let $S_{i}$ be the closure of the connected component of $D \setminus \left(P_{1} \cup \dots \cup P_{2\ell}\right)$ meeting both $P_{i}$ and $P_{i+1}$, then $S_{i}$ is a nice attracting sector for odd $i$ and a nice repulsive sector for even $i$.
\end{enumerate}
\end{prop}

The next statement  takes into account a possible permutation of the petals.
\begin{prop}
\label{prop.carac2}
Let $f\in \hp$ and suppose that for some positive integer $n_{0}$ the map   $f^{n_{0}}$ is locally topologically  conjugate to
a parabolic homeomorphism.
Then so is $f$.
\end{prop}

The proofs are delayed until section~\ref{ss.proof}.

\section{Proof of the theorem}
From now on, $f$ denotes an orientation preserving homeomorphism of the plane that fixes $0$ and satisfies the short trip property.
We fix some open neighbourhood $V$ of $0$ as in the definition of the short trip property, and define the sets
$$
W^s(V) = \bigcap _{n \geq 0} f^{-n}(V) \ \  \mbox{ and } \ \ W^u(V) = \bigcap _{n \leq 0} f^{-n}(V).
$$
Since our hypothesis is symmetric in time both sets share the same properties,
and we will usually restrict the study to $W^s(V)$.

\subsection{Orbits}\label{ss.orbits}

The following lemma shows in particular that the orbits of points near $0$ can only converge
to $0$ or escape from the neighbourhood $V$.
Note that this lemma still holds in any dimension.

\begin{lemma}\label{lem.orbits}~

\begin{enumerate}
\item For every compact subset $K$ of $W^s(V) \setminus \{0\}$, the sequence $(f^n(K))_{n\geq 0}$  converges to $0$.
\item The set $W^s(V)\setminus \{0\}$ 
is open.
\item  The set $W^s(V) \cup W^u(V)$ is a neighbourhood of $0$.
\item If $W^s(V)$  is a  neighbourhood of $0$ then 
$f$ is locally topologically conjugate to $z \mapsto \frac{1}{2}z$; 
 if $W^u(V)$  is a  neighbourhood of $0$ then 
$f$ is locally topologically conjugate to $z \mapsto 2z$. 
\end{enumerate}
\end{lemma}

\begin{proof}
Let $K$ be a compact subset of  $W^s(V)\setminus \{0\}$, and let $W$ be some neighbourhood of $0$ disjoint from $K$. Note that by definition of $W^s(V)$, for every positive $n$, $f^n(K) \subset V$. Now let $N_{W}$ be given by the short trip property. 
Then the property forces  $f^n(K) \subset W$ for every  $n > N_{W}$.
This proves item 1 of the lemma.

Let $x\neq0$ be some point in $W^s(V)$, and $W$ a neighbourhood of $0$ whose closure  does not contain $x$, and such that $W \cup f(W) \subset V$. Let $N_{W}$ be given by the short trip property. 
Let 
$$
O = \left( \bigcap_{n=0}^{N_{W}} f^{-n}(V) \right) \setminus \adhe(W).
$$
This is an open set that contains $x$. We prove item 2 of the lemma by showing that $O$ is included in $W^s(V)$. To see this, let $y \in O$. By definition of $N_{W}$ in the short trip property we have $f^{N_{W}}(y) \in W$. Then we claim that $f^n(y) \in W$ for every $n \geq W$, which will entail $y \in W^s(V)$ as wanted. Assume by contradiction that our claim does not hold and
let $n_{0}$ be the least  integer after $N_{W}$ such that $f^{n_{0}}(y) \not \in W$. Since $f(W) \subset V$ we have $f^{n_{0}}(y) \in V$. The segment of orbit $y,  \dots , f^{n_{0}}(y)$ contradicts the definition of $N_{W}$. This completes the proof of item 2.

We consider again  a neighbourhood $W$ of $0$ such that $W \cup f(W) \cup f^{-1}(W) \subset V$ and $N_{W}$ given by the short trip property. We define the following neighbourhood of $0$,
$$
Z = \bigcap_{n=-N_{W}}^{N_{W}} f^{-n}(W).
$$
We prove by contradiction that $Z \subset W^s(V) \cup W^u(V)$.
Assume some point $x \in Z$ does not belong to $W^s(V)$ nor to $W^u(V)$. Since $W \subset V$,
 the orbit $(f^n(x))$ of $x$ leaves $W$ both in the past and in the future; but by definition of $Z$ this cannot happen for $n$ between $-N_{W}$ and $N_{W}$. Let $r,s$ be the least positive integers such that the points $f^{-r}(x)$ and $f^s(x)$ do not belong to $W$; since $f(W) \cup f^{-1}(W) \subset V$  both points belong to $V \setminus W$ and again we have found a segment of orbit of length $r+s > 2N_{W}$ contradicting the definition of $N_{W}$.

Finally we notice that item 4 is a consequence of item 1 and the   topological characterisation of contractions (proposition~\ref{prop.contraction} above).
\end{proof}

\subsection{Construction of the  petals}\label{ss.petals}
We still consider a homeomorphism $f \in \hp$ satisfying the short trip property, 
and from now on we assume that $f$ is not locally conjugate to the contraction $z \mapsto \frac{1}{2}z$
nor to the dilatation $z \mapsto 2z$. We aim to prove that $f$ is locally conjugate to a  parabolic homeomorphism by ultimately applying propositions~\ref{prop.carac1} and~\ref{prop.carac2}. The main task will be to contruct the family of periodic petals. As a first approximation we will select a finite number of connected components of 
$W^s(V) \cap W^u (V) \setminus \{0\}$, hoping to find one  petal inside each of these components.

We fix an open neighbourhood $V$ of $0$ as before, and we assume $V$ is simply connected.
 According to item 3 of the previous lemma, we can choose a topological closed  disc $D$ which is a neighbourhood of $0$ and included in $W^s(V) \cup W^u(V)$. According to item 4, since we excluded the cases of contractions and dilatations, 
$D$ is not included in $W^s(V)$ nor in $W^u(V)$.
By compactness we can decompose $\partial D$ as the concatenation of $2\ell \geq 2$ arcs $\alpha_{1}, \dots , \alpha_{2\ell}$ such that $\alpha_{i}$ is included in $W^s (V)$ for odd $i$ and in $W^u(V)$ for even $i$. We make the following minimality hypothesis:
 \emph{the number $\ell$ is minimal among all such choices of topological closed  discs  $D$ and decompositions of $\partial D$.}

 For every $i$ (integer modulo $2\ell$) the common endpoint point $x_{i}$  of $\alpha_{i-1}$ and  $\alpha_{i}$ belongs to  $W^s(V) \cap W^u (V)$. We denote by $\cC_{i}$ the connected component of $W^s(V) \cap W^u (V) \setminus \{0\}$ that contains $x_{i}$.
According to item 2 of lemma~\ref{lem.orbits}, this set $\cC_{i}$ is open.
Let $D'$ be a topological closed disc; since $V$ is simply connected,
if $\partial D' \subset V$ then $D' \subset V$. Applying this to the iterates of $D'$, we see that
if $\partial D' \subset  W^s(V) \cap W^u (V) \setminus \{0\}$, then  $D' \subset W^s(V) \cap W^u (V)$. Since $W^s(V)$
is not a neighbourhood of $0$,  we get the following consequence. 
\begin{lemma}\label{lem.simply-connected}
Any connected component of $W^s(V) \cap W^u (V) \setminus \{0\}$ is open and simply connected. In particular, the sets $\cC_{i}$ are 
homeomorphic to the plane.
\end{lemma}

The next lemma is the fundamental step in the construction of the periodic  petals. No dynamics is involved here; indeed, we will only need properties 2 and 3 from lemma~\ref{lem.orbits} on the topology of $W^s(V)$ and $W^u(V)$.
\begin{lemma}\label{lem.closure}
For every $i$, the closure of $\cC_{i}$ contains the fixed point $0$.
\end{lemma} 


\begin{proof}
For notational simplicity we assume $i=1$, and we note  $\cC=\cC_{1}$ and  $x=x_{1} \in \alpha_{2\ell} \cap \alpha_{1}$. Using Schoenflies theorem, up to a change of coordinates, we can assume that $D$ is a euclidean closed disc.

We will argue by contradiction. Assuming that $0$ does not belong to the closure of $\cC$,
we will construct a simple arc
 $\alpha$ with the following properties (we denote by $\partial \alpha$ the set of endpoints of $\alpha$ and set $\inte(\alpha)=\alpha \setminus \partial \alpha$):
\begin{enumerate}
\item $\inte(\alpha) \subset \inte(D)$, $\partial \alpha \subset \partial D$;
\item $\alpha$ separates\footnote{A set $A$ \emph{separates} two points  in a set $B$ if the two points belong to distinct
connected components of $B \setminus A$.} $x$ from $0$ in  $D$;
\item either $\alpha \subset W^s(V)$ and  $\partial \alpha \cap W^u(V) = \emptyset$,

or $\alpha \subset W^u(V)$ and  $\partial \alpha \cap W^s(V) = \emptyset$.
\end{enumerate}

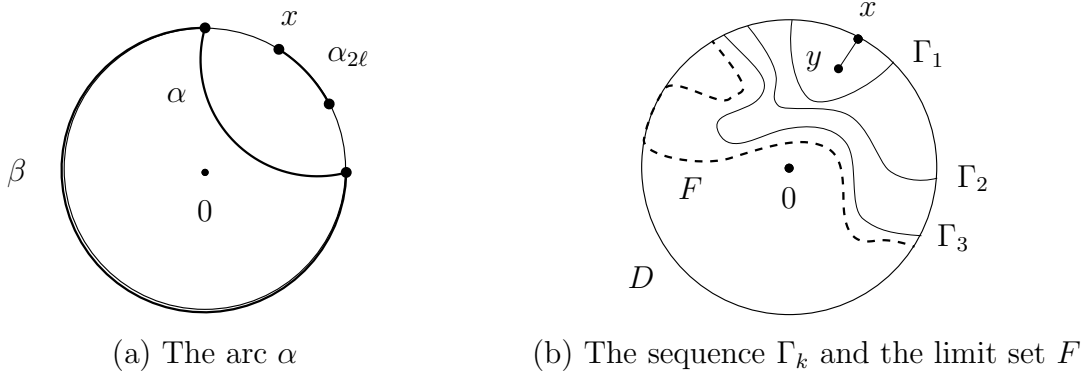
\begin{figure}[htbp]
\begin{center}
\ifx\JPicScale\undefined\def\JPicScale{1}\fi
\psset{unit=\JPicScale mm}
\psset{linewidth=0.3,dotsep=1,hatchwidth=0.3,hatchsep=1.5,shadowsize=1}
\psset{dotsize=0.7 2.5,dotscale=1 1,fillcolor=black}
\psset{arrowsize=1 2,arrowlength=1,arrowinset=0.25,tbarsize=0.7 5,bracketlength=0.15,rbracketlength=0.15}
\begin{pspicture}(0,0)(125,59.82)
\psdots[linewidth=0.1](25.03,29.26)
(25.03,29.26)
\rput{90}(25.03,29.74){\psellipse[linewidth=0.15](0,0)(18.78,18.78)}
\rput{90}(25.03,29.26){\parametricplot[arrows=*-*]{31.61}{61.7}{ t cos 19.26 mul t sin -18.77 mul }}
\psdots[](25.03,48.52)
(43.81,29.26)
\rput{0}(39.97,44.3){\parametricplot[arrows=-]{164.22}{284.33}{ t cos 15.52 mul t sin 15.52 mul }}
\rput{0}(24.88,29.59){\parametricplot[arrows=-]{89.54}{359.01}{ t cos 18.93 mul t sin 18.93 mul }}
\rput(25.03,24.12){$0$}
\rput(0,29.26){$\beta$}
\rput(21.28,39.53){$\alpha$}
\rput(36.3,49.8){$x$}
\rput[l](41.31,44.67){$\alpha_{2\ell}$}
\rput(52.57,24.12){}
\rput{0}(102.69,29.99){\psellipse[linewidth=0.15](0,0)(19.69,19.69)}
\psdots[linewidth=0.2](111.88,47.03)
(111.88,47.03)
\rput(113.19,51){$x$}
\psdots[linewidth=0.2](102.69,29.84)
(102.69,29.84)
\psdots[linewidth=0.2](111.88,47.03)
(109.25,43.06)
\pscustom[linestyle=dashed,dash=1 1]{\psbezier(92.45,46.77)(99.01,38.83)(96.12,40.42)(94.81,39.1)
\psbezier(94.81,39.1)(93.5,37.77)(87.33,42.4)(86.02,40.29)
\psbezier(86.02,40.29)(84.71,38.17)(83.79,35.92)(83.39,33.67)
\psbezier(83.39,33.67)(83,31.43)(91.01,29.97)(94.81,31.16)
\psbezier(94.81,31.16)(98.62,32.35)(102.95,33.81)(105.97,33.15)
\psbezier(105.97,33.15)(108.99,32.48)(109.91,31.03)(109.91,25.74)
\psbezier(109.91,25.74)(109.91,20.45)(111.88,19.79)(114.11,20.18)
\psbezier(114.11,20.18)(116.34,20.58)(119.36,19.39)(119.36,19.39)
}
\psline[linewidth=0.1](111.88,47.03)(109.25,43.06)
\pscustom[linewidth=0.1]{\psbezier(103.08,49.55)(102.43,45.18)(103.48,40.68)(105.31,39.1)
\psbezier(105.31,39.1)(107.15,37.51)(113.58,40.42)(116.6,43.99)
}
\pscustom[linewidth=0.1]{\psbezier(94.02,47.56)(97.96,43.46)(100.59,41.74)(98.88,39.36)
\psbezier(98.88,39.36)(97.18,36.98)(91.79,37.51)(93.37,34.47)
\psbezier(93.37,34.47)(94.94,31.43)(98.88,34.34)(102.03,35)
\psbezier(102.03,35)(105.18,35.66)(108.59,35.53)(110.3,33.15)
\psbezier(110.3,33.15)(112.01,30.76)(110.56,24.68)(112.53,23.09)
\psbezier(112.53,23.09)(114.5,21.51)(116.47,21.11)(120.28,20.85)
}
\pscustom[linewidth=0.1]{\psbezier(97.18,48.75)(99.54,45.84)(101.11,43.2)(100.98,41.48)
\psbezier(100.98,41.48)(100.85,39.76)(100.32,39.1)(101.24,38.17)
\psbezier(101.24,38.17)(102.16,37.25)(107.54,38.44)(110.43,36.98)
\psbezier(110.43,36.98)(113.32,35.53)(114.5,30.76)(115.94,29.57)
\psbezier(115.94,29.57)(117.39,28.38)(120.01,27.85)(122.38,28.52)
}
\rput[l](119.09,45.05){$\Gamma_1$}
\rput[l](125,28.52){$\Gamma_2$}
\rput[l](122.38,20.58){$\Gamma_3$}
\rput(89.56,27.19){$F$}
\rput(102.69,25.87){$0$}
\rput(83,15.29){$D$}
\rput(105.97,44.12){$y$}
\rput(25,5){(a) The arc $\alpha$}
\rput(105,5){(b) The sequence $\Gamma_{k}$ and the limit set $F$}
\rput(0,50){}
\end{pspicture}
\caption{Proof of lemma~\ref{lem.closure}}
\label{fig3}
\end{center}
\end{figure}

From this we will get a contradiction as follows (see figure~\ref{fig3}, (a)).
Assume for example that the first case of the last item holds.
Let $1  \leq i_{1} \leq i_{2} \leq 2 \ell$ be such that the endpoints of $\alpha$ are respectively included in $\alpha_{i_{1}}$ and $\alpha_{i_{2}}$. 
 Since $\partial \alpha$ does not meet $W^u(V)$, both $i_{1}$ and $i_{2}$ are odd, and  in particular  $1 \leq i_{1} \leq i_{2} < 2\ell$.
Let $\beta \subset \partial D$ be the arc with the same endpoints as $\alpha$ and not containing $x$: 
then $\beta$ is covered by $\alpha_{i_{1}} \cup \cdots \cup \alpha_{i_{2}}$, and from the second point we see that 
the Jordan curve $\alpha \cup \beta$ surrounds $0$.
Since $\alpha \cup \alpha_{i_{1}} \cup \alpha_{i_{2}}$ is included in $W^s(V)$,  we can write $\alpha \cup \beta$ as the concatenation of  $i_{2}-i_{1}$ arcs, each included in $W^s(V)$ or $W^u(V)$. This contradicts the minimality hypothesis on $\ell$ since $i_{2} -i_{1} < 2\ell$.

We now assume that $0 \not \in \adhe(\cC)$ and
turn to the construction of such an arc $\alpha$.
According to lemma~\ref{lem.simply-connected}, there exists a homeomorphism $\Phi : \bbR^2 \ra \cC$. Let $(D_{k})$ be the sequence of images under $\Phi$ of the concentric discs with radius $k$ and centre $\Phi^{-1}(x)$. Thus:
\begin{itemize}
\item $x \in D_{1}$,
\item $D_{k} \subset \inte(D_{k+1})$,
\item $\cup_{k \geq 0} D_{k} = \cC.$
\end{itemize}

Let $y$ be some point in $\inte(D)$ sufficiently near $x$ so that the segment $[xy]$ is included in $D_{1} \cap D$.
For every $k$, the set $\partial D_{k} \cap \inte(D)$ is closed in $\inte(D)$
and separates $y$ from $0$ in $\inte(D)$ since $0 \not \in D_{k}$. According to theorem~ V.14.3 in ~\cite{New},
there exists a connected component  of $\partial D_{k} \cap \inte(D)$
that also separates $y$ from $0$. Let $\Gamma_{k}$ denotes the closure of this component;
 thus $\Gamma_{k}$ is a sub-arc of the Jordan curve $\partial D_{k}$ with $\inte(\Gamma_{k}) \subset \inte(D)$, $\partial \Gamma_{k} \subset \partial D$,
 it separates $x$ from $0$ in $D$: in other words it satisfies the first two above properties required for the arc $\alpha$.

Remember that the space of compact connected subsets of $D$ is compact under the Hausdorff metric.
Thus,  up to extraction, we can assume that the sequence  $(\Gamma_{k})$ converges  to a compact connected set $F \subset D$  (see figure~\ref{fig3}, (b)). Since $\Gamma_{k} \subset \partial D_{k}$, the set $F$ is included in $\partial \cC$. By assumption $\partial \cC$ does not contain $0$, so nor does the set $F$.
Then again $F$ separates $y$ from $0$ in $\inte(D)$: if not, there would exist an arc in $\inte(D)$ from $y$ to $0$ missing $F$, but then this arc would also miss $\Gamma_{k}$ for sufficiently big $k$, contrary to the property that $\Gamma_{k}$ separates $y$ from $0$.

 Since $\cC$ is a connected component of the open set $W^s(V) \cap W^u(V) \setminus\{0\}$, its boundary $\partial \cC$  is disjoint from this set. Thus $F$ is disjoint from 
$W^s(V) \cap W^u(V)$;  and since it is included in $D$ it is covered by the two open sets 
$W^s(V) \setminus \{0\}$ and $W^u(V) \setminus \{0\}$. Since $F$ is connected it must be included in one of these two sets, and disjoint from the other one.

To fix ideas suppose that $F \subset W^s(V) \setminus W^u(V)$.
Then for $k$ large enough the arc $\Gamma_{k}$ is also included in $W^s(V)$.
Now this arc almost satisfies the three above properties required for the arc $\alpha$, it only fails to have its endpoints outside $W^u(V)$.
To remedy this we notice that, up to extraction, the two sequences of endpoints $(\Gamma_{k}(0))$, $(\Gamma_{k}(1))$ converges to some points in $z_{0}, z_{1} \in F \cap \partial D$.
Since $F \subset W^s(V) \setminus\{0\}$ we can choose $\varepsilon>0$ so that the euclidean balls $B_{0}$, $B_{1}$ of radius $\varepsilon$ and respective centres $z_{0},z_{1}$ are included in $W^s(V)$. For $k$ large enough $\Gamma_{k}$ meets both balls,
and then we construct the wanted arc $\alpha$ by modifying $\Gamma_{k}$ near its endpoints: we replace two small extreme sub-arcs of $\Gamma_{k}$, respectively  included in $B_{0}$ and $B_{1}$,  by two  euclidean segments reaching the points $z_{0}$ and $z_{1}$. Note that  since $D$ is a euclidean disc both segments, apart from their endpoints $z_{0},z_{1}$, are included in $\inte(D)$.
The endpoints $z_{0},z_{1}$ of the resulting arc $\alpha$ are in $F$, thus outside $W^u(V)$, and $\alpha$ fulfils the third property while still satisfying the first two.
As we explained at the beginning of the proof, the existence of $\alpha$ contradicts the minimality of $\ell$.
\end{proof}

\subsection{Periodicity of the petals}
Unfortunately, we are not able to prove directly that the sets $\cC_{i}$ of the previous section
are periodic for $f$. To overcome this difficulty we will consider slightly larger sets $\cC'_{i}$ 
which will turn out to be periodic. In the next section we will find a periodic petal inside each 
set $\cC'_{i}$.

We suppose that the closure of $V$ is included in some neighbourhood $V'$ of $0$
which still  satisfies the short trip property. In other words, we apply the results of the
 previous sections with $V$ small enough to meet this new assumption.
We note that lemmas~\ref{lem.orbits} and~\ref{lem.simply-connected} apply to $V'$.
Obviously the inclusions $W^s(V) \subset W^s(V')$ and $W^u(V) \subset W^u(V')$ hold.
Let the disc $D$ and the sets $\cC_{i}$ be defined from $V$ as in the previous section.
Each set $\cC_{i}$ is connected and included in $W^s(V') \cap W^u(V') \setminus\{0\}$,
and thus it is included in one connected component of $W^s(V') \cap W^u(V') \setminus\{0\}$
which we denote by $\cC'_{i}$.
\begin{lemma}
The sets $\cC'_{i}$ are periodic: 
for every $i$ there exists some positive integer $q_{i}$ such that $f^{q_{i}}(\cC'_{i})=\cC'_{i}$.
\end{lemma}
\begin{proof}
We first note that the set  $W^s(V') \cap W^u(V') \setminus\{0\}$ is invariant under $f$,
and hence for every $n$ the set $f^n(\cC'_{i})$ is a connected component of $W^s(V') \cap W^u(V') \setminus\{0\}$.

We  claim that for every $i$ there exist infinitely many $n$ such that $f^n(\cC_{i})$ meets the circle $\partial D$.
Assuming the claim, we choose some limit point $x \in \partial D$ of the sequence $(f^n(\cC_{i}))_{n \in \bbZ}$.
Since the point  $x$ is a limit of points whose whole orbits are included in $V$, its orbit is included in $\adhe(V) \subset V'$:
in other words $x$ belongs to $W^s(V') \cap W^u(V') \setminus\{0\}$. Let $O$ be the connected component of this last set containing $x$.
According to lemma~\ref{lem.simply-connected}, $O$ is open, and thus there exists infinitely many integers $n$ such that $f^n(\cC_{i})$ meets $O$.
For every such integer $n$, the set $f^n(\cC'_{i})$ is a connected component of $W^s(V') \cap W^u(V') \setminus\{0\}$ that meets $O$,
thus it coincides with $O$. Thus we find two integers $n_{1} < n_{2}$ such that $f^{n_{1}}(\cC'_{i})=f^{n_{2}}(\cC'_{i})$, which proves that $\cC'_{i}$ is periodic.

We prove the claim. By lemma~\ref{lem.closure} the fixed point $0$ belongs to the closure of $\cC_{i}$.
Since $\cC_{i}\cup \{0\}$ is not a neighbourhood of $0$, this  point  also belongs to the closure of $\partial \cC_{i}$.
Furthermore, 
$$
\begin{array}{rcl}
\left( \partial \cC_{i} \right) \setminus \{0\}   &  \subset &    \partial \left( W^s(V) \cap W^u(V)  \right)  \setminus \{0\}  \\
  & \subset & \left( W^s(\adhe V) \cap W^u(\adhe V) \right) \setminus  \left( W^s(V) \cap W^u(V) \right).  
\end{array}
$$
Consequently for any $z \in \partial \cC_{i}  \setminus \{0\} $ there exists an integer $n$ such that $f^n(z) \in \partial V$.
Let $(z_{k})$ be a sequence in $\partial \cC_{i}$ converging to $0$, then  any sequence $n_{k}$
such that $f^{n_{k}}(z_{k}) \in \partial V$ is unbounded, 
because the union of finitely many iterates of $\partial V$ is a closed set which does not contain $0$.
For any $k$ the set $f^{n_{k}}(\cC_{i})$ is connected, its closure contains $0$ and meets $\partial V$, thus it also meets $\partial D$.
This completes the proof of the claim.
\end{proof}

\subsection{Construction of the local conjugacy}
We finally define the petals. 
According to the previous lemma we can choose some $n_{0}>0$ such that $F=f^{n_{0}}$ leaves invariant every set $\cC'_{i}$. In view of proposition~\ref{prop.carac2}, theorem~\ref{theo.main}  will follow from the fact that $F$ is locally conjugate to a locally holomorphic parabolic homeomorphism. 
Let us prove this fact.

Recall that $\cC'_{i}$ is homeomorphic to the plane (lemma~\ref{lem.simply-connected}), and 
for any compact set $K \subset \cC'_{i}$, the sequence $(f^{n}(K))_{n \geq 0}$ converges to $0$ (lemma~\ref{lem.orbits}).
Consequently   theorem~\ref{theo.kerek} tells us that for every $i$ the restriction of $F$ to the invariant set $\cC'_{i}$ is conjugate to the plane translation $z \mapsto z+1$. Picking a horizontal line and bringing it back  under the conjugacy, we see that
 the point $x_{i}$ of $\cC'_{i}$ is on  a topological line $\Delta_{i} \subset \cC'_{i}$ such that  $F(\Delta_{i})=\Delta_{i}$.
Let $P_{i}$ be the closed topological disc bounded by the curve $\Delta_{i}\cup\{0\}$ (here we use Schoenflies theorem).
It is clear that  $P_{i}$ is a regular  invariant petal for $F$.

The curve $\alpha_{i}$, defined at the beginning of section~\ref{ss.petals}, meets both petals $P_{i}$ and $P_{i+1}$.
Furthermore for odd $i$ we have  $\alpha_{i} \subset W^s(V)$ so by lemma~\ref{lem.orbits} the sequence $(F^n(\alpha_{i}))_{n \geq 0}$  converges to $0$, and the sequence $(F^n(\alpha_{i}))_{n \leq 0}$  converges to $0$ for even~$i$. Thus the construction of a local conjugacy between $F$ and $z \mapsto z(1+z^\ell)$ now boils down to the following lemma.

\begin{lemma}[see figure~\ref{fig1}]
\label{lemma.carac1bis}
Let $f \in \hp$. Fix some integer $\ell \geq 1$. Assume the following hypotheses.
\begin{enumerate}
\item There exist $2\ell$ regular invariant petals $P_{1}, \dots , P_{2\ell}$,
whose pairwise intersections are reduced to $\{0\}$.
\item There exists a topological closed disc $D$ which is a neighbourhood of $0$, 
and whose boundary is the concatenation of $2\ell$ arcs $\alpha_{1}, \dots , \alpha_{2\ell}$,
each arc $\alpha_{i}$ having one endpoint on $P_{i}$ and the other one on $P_{i+1}$.
\item For odd $i$ the sequence $(f^n(\alpha_{i}))_{n \geq 0}$ converges to $0$, and for even $i$
the sequence $(f^n(\alpha_{i}))_{n \leq 0}$ converges to $0$.
\end{enumerate}
 Then $f$ is locally conjugate to $z \mapsto z(1+z^{\ell})$.
\end{lemma}
Note that we do not suppose that $\partial D \cap P_{i}$ is a connected set, nor that $\alpha_{i}$  does not meet some petal $P_{j}$ with $j \neq i,i+1$, contrarily to proposition~\ref{prop.carac1}.
 An important step of the proof will be to check  the not obvious fact that the petals indexation coincide with their cyclic order around $0$.

\begin{proof}[Proof of lemma~\ref{lemma.carac1bis}]
 Consider a homeomorphism  $f\in \hp$ satisfying the hypotheses of the lemma.
The arc $\alpha_{i}$  contains a minimal  sub-arc $\alpha'_{i}$ connecting $P_{i}$ to $P_{i+1}$:
the endpoints of $\alpha'_{i}$  are respectively on  $P_{i}$ and $P_{i+1}$, and its interior $\inte(\alpha'_{i})$ is disjoint from $P_{i}$ and $P_{i+1}$.
Let $i$ be odd, so that the sequence $(f^n(\alpha'_{i}))_{n \geq 0}$ converges to $0$.
Then we define an attracting sector $S'_{i}$ as follows.
We consider the curve obtained by concatenating 
 the arc $\alpha'_{i}$, the sub-arc of $\partial P_{i}$ from the endpoint of $\alpha$ to $0$ following the dynamical orientation of $\partial P_{i}$, and the similar sub-arc on $P_{i+1}$ (see remark~\ref{rem.orientation}).
This curve is clearly a Jordan cuve, it bounds the topological closed disc $S'_{i}$.
It is not difficult to see that $S'_{i}$ is indeed an attracting sector. We apply item 1 of claim~\ref{claim2} to get a nice attracting sector $S_{i} \subset S'_{i}$ having the same germ as $S'_{i}$ at $0$.
For even $i$ we   symmetrically define a repulsive sector $S'_{i}$ and a nice repulsive sector $S_{i}$.

Since the petals are topological closed discs whose pairwise intersection is reduced to $\{0\}$, Schoenflies theorem can be used to prove that the union of the petals is homeomorphic to the model pictured on the left side of  figure~\ref{fig1}; but 
we still have to prove that their cyclic order is as shown on the right side of the figure (or the reverse one).
For this we argue by contradiction. Suppose there exists some $i$ such that the petals $P_{i}$ and $P_{i+1}$ are not adjacent: they are locally separated near $0$  by the union of the other petals. Then there exists another petal $P_{j}$ such that the sector
  $S_{{i}}$  contains a neighbourhood of $0$ in $P_{j}$; in other words the set
 $\adhe(P_{j} \setminus S_{{i}})$ is a compact subset of $P_{j} \setminus \{0\}$.
 Using claim~\ref{claim1} that describes the dynamics of $f$ on $P_{j}$,
we find a point $x \neq 0$ whose full orbit $\{f^n(x),n \in \bbZ\}$ is included in $P_{j} \cap S_{i}$. But due to item 2 of claim~\ref{claim2} 
 a nice attractive or repulsive sector  contains no full orbit, which provides the contradiction.

Up to reversing the indexation, we may now assume that the petals are indexed in the positive cyclic order around $0$ (so that Schonflies theorem provides an \emph{orientation preserving} homeomorphism that sends each $P_{i}$ on the model of figure~\ref{fig1}).
Note that the same argument as above, applied to $P_{i}$ and $P_{i+1}$ instead of $P_{j}$,
 shows that the interior of $S_{i}$ is disjoint from $P_{i}$ and $P_{i+1}$,
and in particular the dynamical order on the boundaries of the petals is as indicated on figure~\ref{fig1}: the petal $P_{i}$ is direct for odd $i$ and indirect for even $i$.

Up to replacing $S_{i}$ with some smaller nice sector, we get that 
\begin{enumerate}
\item for any $i,j$ with $j \neq i,i+1$, we have  $P_{i} \cap S_{j}=\{0\}$;
\item for any $i \neq j$, we have $S_{i} \cap S_{j}=\{0\} = f^{-1}(S_{i}) \cap S_{j}$.
\end{enumerate}
Consider the set $D = P_{1} \cup S_{1} \cup \cdots \cup P_{2\ell} \cup S_{2\ell}$.
Thanks to item 2 the maximal invariant set of $D$ is the union of the petals $P_{i}$.
Thus $D$ is a topological closed disc  satisfying the hypotheses of proposition~\ref{prop.carac1}. Now  the lemma follows from the proposition.
\end{proof}

\section{Proof of propositions~\ref{prop.carac1} and~\ref{prop.carac2}}
\label{ss.proof}
\begin{proof}[Proof of proposition~\ref{prop.carac1}]
The fact that for the map $z \mapsto z(1+z^{\ell})$ there exists a topological closed disc $D$ satisfying properties 1,2,3 of the proposition is  part of the proof of Camacho-Leau-Fatou theorem
(see~\cite{Cam,Mil}).

We turn to the proof of the reverse implication.
We consider a homeomorphism  $f\in \hp$ and a disc $D$ satisfying properties 1,2,3 of the proposition.
We have to prove that if $f'\in \hp$ and a disc $D'$  satisfies the same properties (with the same number $\ell$) then $f$ and $f'$ are locally topologically conjugate. Note that the union of all sectors $S_{i}$ and petals $P_{i}$ is   equal to $D$.

Let $i$ be an odd integer. Since $S_{i}$ is an attracting sector between $P_{i}$ and $P_{i+1}$, the petal $P_{i}$ is direct, while the petal $P_{i+1}$ is indirect (see remark~\ref{rem.orientation}).
The same is true for $f'$. Thus  according to claim~\ref{claim1}, the restriction of $f$ and $f'$ to $P_{i}$ and $P'_{i}$ are conjugate. The conjugacies can be glued together to obtain an orientation preserving homeomorphism $\Phi : \cup P_{i} \ra \cup P'_{i}$ which sends $P_{i}$ onto $P'_{i}$ and  is a conjugacy between the restrictions of $f$ and $f'$.
 

The image under $\Phi$ of $S_{i}\cap \left(P_{i}\cup P_{i+1}\right)$ is not necessarily equal to $S'_{i}\cap \left(P'_{i}\cup P'_{i+1}\right)$.
But  using item~2 of claim~\ref{claim2} we  can replace $S_{i}$ and $S'_{i}$
with smaller nice attracting sectors so that this equality becomes true (see item 2 of remark~\ref{rema.sector}).
We can now use item~3 of remark~\ref{rema.sector} to
extend $\Phi$ to a homeomorphism between $D$ and $D'$, sending $S_{i}$ onto $S'_{i}$ 
and conjugating the restrictions of $f$ and $f'$. We do this for every attracting or repulsive sectors $S_{i}$.
We further extend $\Phi$ to a homeomorphism of the plane.
The conjugacy relation $f' \Phi= \Phi f$ is satisfied on $D\cap f^{-1}(D)$.
This completes the proof of the proposition.
\end{proof}

To prove proposition~\ref{prop.carac2} we need a claim.
\begin{claim}\label{claim3}
Let $Q_{1},Q_{2}$ be two invariant petals included in a regular invariant petal $P$ for $F \in \hp$. Then $Q_{1}$ meets $Q_{2}$, and there exists a unique connected component $O$ of $\inte (Q_{1}) \cap \inte (Q_{2})$ such that $F(O) = O$. Furthermore, the closure of $O$ is a regular invariant petal for $F$.
\end{claim}
\begin{proof}
The first part is easily proved using the translation model given by claim~\ref{claim1}.
The only difficulty in the second part consists in checking that the closure of $O$ is indeed a topological closed disc. But this follows from a previously quoted result of K\'er\'ekj\'art\'o (\cite{Ke2, LCY}).
\end{proof}
Also note that if $Q \subset P$ are two regular invariant petals and $P$  is direct then $Q$ is direct.

\begin{proof}[Proof of proposition~\ref{prop.carac2}]
Let $f^{n_{0}} = F$ be conjugate to a  parabolic homeomorphism $F_{0}$. 
Up to increasing $n_{0}$, we can assume that 
$F_{0}'(0) = 1$, and thus  $F$ is locally conjugate to $z \mapsto z(1 + z^{\ell})$ for some integer $\ell$.
Let $D$ be a nice disc for $F$, and 
let  $\{P_{1}, \dots , P_{2\ell}\}$ be the family of petals associated with $D$, as given by proposition~\ref{prop.carac1}.

For each $i$ we choose a small invariant petal  $Q_{i}$ for $F$ included in $P_{i}$.
Since $f^{n_{0}}=F$ and $Q_{i}$ is invariant for $F$, if $Q_{i}$ is small enough then every iterates  $f^n(Q_{i})$ is included in $D$. Since $f^n(Q_{i})\setminus\{0\}$ is connected and invariant for $F$, it is included in a connected component of the $F$-maximal invariant set of $D \setminus \{0\}$, that is, 
$f^n(Q_{i})$ is included in some petal $P_{j}$.
Let us fix $j$ and consider the finite family of all the petals $f^n(Q_{i})$ for  $n \in \bbZ, i \in \bbZ/2\ell\bbZ$ which are included in $P_{j}$. We denote the intersection of their interiors by $O_{j}$.
Applying claim~\ref{claim3} inductively we see that the closure of $O_{j}$ is a regular invariant petal for $F$, let us call it $\bar{P_{j}}$.

By construction the petals in the family $\{\bar{P_{j}}\}$ are permuted by $f$, their pairwise intersections are reduced to $0$, and their cyclic order around $0$ is given by the cyclic order on the indices $i \in \bbZ/2\ell\bbZ$.
Since $f$ is an orientation preserving homeomorphism there exists $i_{0}$ such that for every $i$, $f(\bar P_{i})=\bar P_{i+i_{0}}$. Furthermore, since $f$ respects the dynamical orders  induced by $F$ on the boundary of the petals, $i_{0}$ must be even.
The order $n'_{0}$ of the permutation $i \mapsto i+i_{0}$ is a divisor of $n_{0}$ (maybe strict).
It is easy to see that there exists  another nice disc $\bar D$ for $F$ whose maximal invariant set is the union of this family of petals. The nice attractive and repulsive sectors $\bar S_{i}$ for $F$, associated with $\bar D$, are clearly also  attractive or repulsive sectors for $f^{n'_{0}}$, and according to claim~\ref{claim2} we can find within each $\bar S_{i}$ a nice attracting or repulsive sector $\bar{\bar{S_{i}}}$ for $f^{n'_{0}}$ having the same germ at $0$. Now the topological closed disc $\bar{\bar{D}}$ defined as the union of all petals $\bar P_{i}$ and sectors $\bar{\bar{S_{i}}}$ is a nice disc for $f^{n'_{0}}$, the hypotheses of proposition~\ref{prop.carac1} are satisfied, and $f^{n'_{0}}$ is conjugate to $F$.

Using these families of petals and sectors we are now in a position to
construct a local conjugacy $\Phi$ between $f$ and the model  map
 $f_{0} : z \mapsto  e^{2i\pi \frac{i_{0}}{2\ell}}z(1 + z^{\ell})$. Note that $f_{0}^{n_{0}'}$ is conjugate to  $z \mapsto z(1 + z^{\ell})$ and that 
 $f_{0}$ permutes a family of regular invariant petals for $f_{0}^{n_{0}'}$.
 The construction of the conjugacy is similar 
 to the one defined in the proof of proposition~\ref{prop.carac1}. Here is the main difference:  
 since the petals are permuted by $f$, we have first to define  a conjugacy $\Phi$ between $f^{n'_{0}}$ and $f_{0}^{n'_{0}}$ on some petal $\bar P_{i}$, and then there is a unique way to  extend it to the $f$-orbit of $\bar P_{i}$ to get a conjugacy between $f$ and $f_{0}$.
 We do the same for every $f$-orbit of petals, and for every $f$-orbit of sector. This completes the proof of proposition~\ref{prop.carac2}.
  \end{proof}
 
%
%
%
%


\end{document}